\documentclass[11pt]{article} 
\usepackage[utf8]{inputenc} 

\usepackage{amsmath}
\usepackage{amsthm}
\usepackage{amsfonts} 
\usepackage{color}
\usepackage{geometry} 
\usepackage{setspace} 
\usepackage{bbold}
\usepackage{booktabs} 
\usepackage{array} 
\usepackage{paralist} 
\usepackage{verbatim} 
\usepackage{subfig}
\usepackage{sectsty}
\usepackage[nottoc,notlof,notlot]{tocbibind} 
\usepackage[titles,subfigure]{tocloft} 
\allsectionsfont{\sffamily\mdseries\upshape}
\usepackage{bm} 
\newtheorem{Prop}{Proposition}[section]
\newtheorem{thm}{Theorem}[section]
\newtheorem{cor}{Corollary}[section]
\newtheorem{lem}[thm]{Lemma}
\newtheorem*{remark}{Remark}
\newtheorem{Def}{Definition}[section]

\newcommand{\beq}{\begin{equation}}
\newcommand{\eeq}{\end{equation}}

\newcommand{\fc}{\frac}
\newcommand{\nid}{\noindent}
\newcommand{\ra}{\rightarrow}
\renewcommand{\l}{\left}
\renewcommand{\r}{\right}
\newcommand{\Z}{\mathbb{Z}}
\newcommand{\N}{\mathbb{N}}
\newcommand{\T}{\mathbb{T}}
\newcommand{\R}{\mathbb{R}}
\renewcommand{\a}{\alpha}

\newcommand{\e}{\epsilon}
\renewcommand{\i}{\mathbf{i}}
\renewcommand{\(}{\left(}
\renewcommand{\)}{\right)}
\newcommand{\D}{D_{\triangle_{\bf{x}}} (\bm{a},\bm{\a}; N)}
\renewcommand{\bf}{\mathbf}

\newcommand{\ba}{\begin{array}}
\newcommand{\ea}{\end{array}}
\renewcommand{\d}{\delta}

\newcommand{\bal}{\begin{aligned}}
\newcommand{\eal}{\end{aligned}}

\renewcommand{\L}{\mathcal{L}}

\newcommand{\Dbar}{\bar{D}_{\triangle_{\bf{x}}} (\bm{a},\bm{\a}; N)}

\newcommand{\nstar}{\bf{n^*}}
\newcommand{\Uqplus}{U(\bf{h}^{q},\e^{+})}
\newcommand{\Uqminus}{U(\bf{h}^{q},\e^{-})}
\newcommand{\Uq}{U(\bf{h}^{q},\e^{\pm})}
\newcommand{\Up}{U(\bf{h}^{p},\e^{\pm})}
\newcommand{\hp}{h^{(p)}}
\newcommand{\hq}{h^{(q)}}
\newcommand{\Ln}{\mathcal{L}(\bf{n})}
\newcommand{\Vol}{\text{Vol}}

\newcommand{\tria}{\triangle_{\bf{x}}}

\newcommand{\V}{\bf{V}}
\newcommand{\W}{\bf{W}}
\newcommand{\n}{\bf{n}}

\newcommand{\bma}{\bm{\alpha}}
\newcommand{\Dbarone}{\bar{D}_{\triangle_{\bf{x}},1} (\bm{a},\bm{\a};N)}
\newcommand{\Dbartwo}{\bar{D}_{\triangle_{\bf{x}},2} (\bm{a},\bm{\a};N)}
\newcommand{\ftwo}{f_2(\bf{n},\bf{x},\bm{a},\bm{\a};N)}
\newcommand{\Dtwoone}{\bar{D}_{2,1}}
\newcommand{\Dtwotwo}{\bar{D}_{2,2}}
\renewcommand{\v}{\bf{v}}
\newcommand{\bfe}{\bf{\epsilon}}

\newcommand{\<}{\langle}
\renewcommand{\>}{\rangle}

\makeatletter 
\@addtoreset{equation}{section}
\makeatother

\title{Absolute Bounds for Ergodic Deviations of Toral Translations Relative to Triangles in $\mathbb{T}^2$}
\author{Hao Wu}
\date{}

\begin{document}

\maketitle
\textbf{Abstract.} Following Beck's work on toral translations relative to straight boxes in $\T^n$, we prove a weaker upper bound and the same lower bound for ergodic discrepancies of toral translations relative to a triangle in $\T^2$. Specifically, given a positive increasing function $\varphi(n)$, we show that for a full measure set of translation vectors $\bm{\a}\in \T^2$, if the series $\sum_{N=1}^{\infty} \fc{1}{\varphi (N)} $ converges, then the maximal discrepancy of toral translations relative to the triangles of a given slope $\tau$ is bounded from above by $Const(\bma,\tau) (\log N)^2 \varphi^2(\log \log N)$, and there would be infinitely many $N$'s such that the maximal discrepancy is greater than $(\log N)^2\varphi(\log \log N)$ if the series $\sum_{N=1}^{\infty} \fc{1}{\varphi (N)} $ diverges. An important difference between our result and that of Beck' is an additional factor $\varphi(\log \log N)$, which is necessary in our proof for controlling the new small divisors created by the hypotenuse.

\clearpage
\tableofcontents
\clearpage
\section{Introduction}
Let $\bma\in \T^d$, the toral translation $T_{\bma}$ over $\T^d=(\R/\Z)^d\cong [0,1)^d$ is defined by $\bf{x}\mapsto \bf{x}+\bma$ for $\bf{x}\in \T^d$, in the sense of modulo 1 for each coordinate. By Weyl's criterion (see \cite{Drmota}, section 1.2.1), the sequence $\{n\bma\}_{1\le n\le N}$ becomes equidistributed over $\T^d$ if the translation vector $\bma$ is irrational, i.e. for any measurable set $B\subset \T^d$, if $1, \ \a_1,\ \dots, \ \a_d$ are rationally independent, then as $N\rightarrow \infty$,
%the ratio of the number of visits inside a measurable set $B$ before time $N$ converges to the volume of the set $\Vol(B)$ as $N\rightarrow \infty$ . 
$$
\fc{\sum_{n=1}^{N}\chi_{B}(n\bma)}{N}\rightarrow \Vol(B).
$$
To measure the rate of convergence, we introduce the discrepancy function defined as the difference between the \emph{actual} number of hits before time $N$ and the \emph{expected} number of hits $N\cdot \Vol(B)$:
$$\sum_{n=1}^{N}\chi_{B}(n\bma)-N\Vol(B).$$

In 1920s, Khintchine \cite{Khintchine} proved that the maximal discrepancy for almost every toral translation in $\T^1$ relative to intervals is between $(\log N)(\log\log N)$ and $(\log N)(\log \log N)^{1+\e}$ for any $\e>0$. His proof used the continued fraction algorithm for irrationals. Due to the absence of the continued fraction algorithm in higher dimensions, the higher-dimensional counterpart of Khintchine's theorem remained unsolved for 70 years. In 1994, by a completely different method, which consists of a combination of Fourier analysis, the ``second-moment method'' and combinatorics, J. Beck\cite{Beck} proved the following precise multidimensional analogue of Khintchine's theorem:
\begin{thm}
Let $k\ge 2$, $\bm{\a}=(\a_1,\dots, \a_k)\in \R^k$ be the translation vector, and $B(\bf{x})=[0, x_1)\times \dots \times [0,x_k)\subset [0,1)^k$, define the ergodic discrepancy:
$$
D(\bm{\a}, \bf{x}; m)=\sum_{1\le n\le m} \chi_{B(\bf{x})}(n\bm{\a}\mod 1)-m\Vol(B(\bf{x}))
$$
and the maximal discrepancy:
$$
\Delta(\bm{\a}; N)= \max_{\mbox{\scriptsize$\begin{array}{c}
1\le m \le N \\
\bf{x}\in [0,1]^k\\\end{array}$}} |D(\bm{\a}, \bf{x}; m)|.
$$
Then for arbitray positive increasing function $\varphi(n)$ of $n$, 
\beq\label{Thm of Beck}
\Delta(\bm{\a}; N)\ll (\log N)^k\cdot \varphi(\log \log N) \Longleftrightarrow \sum_{n=1}^\infty \fc{1}{\varphi(n)}<\infty
\eeq
for almost every $\bm{\a}\in \T^k$, where $\ll$ denotes the Vinogradov symbol, e.g. $f(N)\ll g(N)$ means that $|f(N)|<c \cdot g(N)$ for all $N$ with a uniform constant $c$, here the constant depends on $\bma$.
\end{thm}

In this paper, following Beck's method, we consider the ergodic discrepancies of irrational toral translations relative to \emph{triangles} inside $\T^2$. Let $\bm{\a}=(\a_1,\a_2)\in[0,1]^2$ be the translation vector, where $1,\a_1,\a_2$ are rationally independent, let $\bm{a}=(a_1,a_2)\in\R^2$ be the starting point, and denote the right triangle of sides $\bf{x}=(x_1,x_2)\in (0,1]^2$ by: 
\beq\label{Definition of the triangle x}
\triangle_{\bf{x}}=\{(y_1,y_2)\ | \ 0\le y_1/x_1+y_2/x_2 <1, \ 0\le y_i\le x_i, i=1,2\}.
\eeq
Define the ergodic discrepancy of the toral translation $\bm{\a}$ starting from the point $\bm{a}$ relative to the triangle $\triangle_{\bf{x}}$:

$$D_{\triangle_{\bf{x}}} (\bm{a},\bm{\a};m) = \sum\limits_{n=1}^{m}  \chi_{\triangle_{\bf{x}}}(\bm{a}+n\bm{\alpha} \mod 1)- m\text{Vol}(\triangle_{\bf{x}})$$

Since the sequence $\{n\bm{\a} \mod 1\}$ is equistributed over $\T^2$, the ``expected number'' of points in the sequence that visit the triangle $\tria$ before time $m$ is $m\text{Vol}(\tria)$, and we would like to prove a similar bound for the error term $D_{\tria}$.

\nid\textbf{Notation.} Through out this paper,  $\| \cdot\|$ denotes the distance to the closest integer. For a given triangle $\tria$ defined as above \eqref{Definition of the triangle x}, let $\tau=x_1/x_2$ denote the slope of the hypotenuse of the triangle. Without loss of generality, we assume that $\tau\in[0,1]$. For $\n=(n_1,n_2,n_3)\in \Z^3$, and a vector $\bma=(\a_1,\a_2)\in \R^2$, let $\n\bma=n_1\a_1+n_2\a_2$ denote the inner product of their first two coordinates. From Section 5, the $\n$ would denote $(n_1, n_2)\in \Z^2$ instead of $(n_1,n_2,n_3)$, and $n_3$ would denote the closest integer to $n_1\a_1+n_2\a_2$, since in Section 4 the other possibilities for $n_3$ have already been estimated.

The main result of this paper is the following:
\begin{thm}{\label{main result}}
Let $\bm{\a}=(\a_1, \a_2)\in \R^2$ be the translation vector, $\bm{a}=(a_1,a_2)\in [0,1]^2$ be the starting point, and for $\bf{x}=(x_1, x_2)\in (0,1]^2$ define the triangle $\triangle_{\bf{x}}=\{(y_1,y_2)\ | \ 0\le y_1/x_1+y_2/x_2 <1, \ 0\le y_i< x_i, i=1,2\}$, define the ergodic discrepancy:
$$D_{\triangle_{\bf{x}}} (\bm{a},\bm{\a}; m) = \sum\limits_{n=1}^{m}  \chi_{\triangle_{\bf{x}}}(\bm{a}+n\bm{\alpha} \mod 1)- m\text{Vol}(\triangle_{\bf{x}})$$
and the maximal discrepancy for triangles of a fixed slope $\tau$:
$$
\Delta(\bm{\a},\tau; N)= \max_{\mbox{\scriptsize$\begin{array}{c}
1\le m \le N; \\
\bm{a}\in [0,1]^2;\\
x_1/x_2=\tau;
\end{array}$}} |D_{\tria} (\bm{a},\bm{\a}; m)|.
$$
Given an arbitray positive increasing function $\varphi(n)$ of $n$, then for almost every $\bm{\a}\in \R^2$ and almost every $\tau\in \R^+$, we have: 
\beq\label{divergent part of the theorem}
\sum_{n=1}^\infty \fc{1}{\varphi(n)}<\infty \ \Rightarrow \ \Delta(\bm{\a},\tau; N)\ll (\log N)^2\cdot \varphi^2(\log \log N),
\eeq
\beq\label{convergent part of the theorem}
\sum_{n=1}^\infty \fc{1}{\varphi(n)}=\infty \ \Rightarrow \ \Delta(\bm{\a},\tau; N) > (\log N)^2\cdot \varphi(\log \log N) \text{ i.o. },
\eeq
where i.o. stands for infinitely often, and the constant in \eqref{divergent part of the theorem} depends on $\bma$ and $\tau$.

\end{thm}
\begin{remark}
Unfortunately the statement does not give a perfect equivalence, the additional factor $\varphi(\log \log N)$ is needed for controlling the extra small divisors that arise from the hypotenuse of the triangle, which does not occur in the case of straight rectangles.
\end{remark}

In fact, after fixing a toral translation, the bound of the discrepancy is uniform for any starting point $\bm{a}$, therefore it is equivalent to translate the right triangle $\tria$ inside $\T^2$, and since every general shaped triangle could be decompsed into a finite sum or difference of right triangles, it is immediate that the theorem above hold true for almost every triangle inside $\T^2$, hence we obtain a corollary:
\begin{cor}\label{Corollary for triangles of general shape}
The same bounds in Theorem \ref{main result} hold for almost every triangle $\in \T^2$ (or more generally almost every polygon), with a constant that depends on the slopes of the sides of the triangle (or the polygon).
\end{cor}
%With Corollary \ref{Corollary for triangles of general shape} and a classical approximation of circle by triangles, we could obtain an almost sure upper bound for the ergodic discrepancy of toral translations relative to a disk inside $\T^2$.
%\begin{thm}\label{almost sure upper bound for disk}
%Let $\bm{\a}=(\a_1, \a_2)$ be the translation vector, $\bm{a}=(a_1,a_2)$ be the starting point, let $\C_r$ be the disk of radius $r\in (0,1/2)$, define the ergodic discrepancy of the toral translation:
%$$D_{\C_{r}} (\bm{a},\bm{\a}; m) = \sum\limits_{n=1}^{m}  \chi_{\C_r}(\bm{a}+n\bm{\alpha} \mod 1)- m\text{Vol}(\C_r)$$
%and the maximal discrepancy:
%$$
%\Delta(\bm{\a}; N)= \sum_{\mbox{\scriptsize$\begin{array}{c}
%1\le m \le N \\
%(\bm{a},r)\in [0,1)^2\times (0,\fc{1}{2})\\\end{array}$}} D_{\C_r} (\bm{a},\bm{\a}; m).
%$$
%Then for almost every $\a\in \R^2$, and every $\e>0$,
%$$
%\Delta(\bm{\a}; N)\ll N^{\fc{1}{3}}(\log N)^{\fc{2}{3}+\e}
%$$
%\end{thm}

This paper is organized as the following: in Section 2, for the convenience of later estimations, we transform the ergodic discrepancy to its Fourier series by using Poisson's summation formula. In Section 3 we estimate the contribution of the "tail" of the Fourier series, i.e. the high frequency nodes or the extremely small divisors. Section 4 and 5 deal with the main part of the discrepancy, Section 4 is about the constant part and Section 5 deals with the exponential part, both of which will be properly defined later. Combining the section 3-5, we have an overall estimation of the discrepancy, which proves our main theorem.

\section{Poisson's summation formula}

In this section, following \cite{Beck}, we use the Poisson formula to transform the ergodic discrepancy into a Fourier series.  The main result of this section is Proposition \ref{Prop Dbar}. It gives the main subject of our analysis: a roof-like average of the original Fourier series of the discrepancy, which has better convergence properties. First we introduce the Fourier series of the ergodic dispcrepancy by the following Lemma.
\begin{lem}{\label{Lemma of Fourier Series}}
By using the Poisson formula, the ergodic sum $D_{\triangle_{\bf{x}}} (\bm{a},\bm{\a}; N)$ becomes:
$$
\begin{aligned}
D_{\triangle_{\bf{x}}} (\bm{a},\bm{\a}; N)
&=\fc{\i^3}{(2\pi)^3} \sum_{\bf{n}\in \Z^3\backslash\{0\}} 
\prod_{j=1}^2 e^{-2\pi \i n_ja_j}\(\fc{1-e^{2\pi \i n_2x_2}}{n_1 n_2}\fc{e^{-2\pi \i (n_1\a_1+n_2\a_2-n_3)N}-1}{n_1\a_1+n_2\a_2-n_3}\)\\
&+\fc{\i^3}{(2\pi)^3} \sum_{\bf{n}\in \Z^3\backslash\{0\}} 
\prod_{j=1}^2 e^{-2\pi \i n_ja_j}\(\fc{e^{2\pi\i n_1x_1}-e^{2\pi\i n_2 x_2}}{n_1(n_1\fc{x_1}{x_2}-n_2)}\fc{e^{-2\pi \i (n_1\a_1+n_2\a_2-n_3)N}-1}{n_1\a_1+n_2\a_2-n_3}\).
\end{aligned}
$$
\end{lem}
\begin{proof}
The ergodic sum could be calculated as follows, note that the condition:
$n\bm{\a} \mod 1 \in \tria$ is equivalent to $\exists \bf{m}=(m_1, m_2, m_3)\in \mathbb{Z}^3$ such that: 

$$
\begin{array}{c} 

0\le m_1\alpha_i  -m_2 <x_i ; \ i=1,2;\\
\sum_{i=1}^{2}\fc{n\a_i}{x_i}< 1; \\
1\le n \le N; \\
\end{array}
$$

Consider the translated lattice in $\mathbb{R}^3$, 
$$ 
L(\bm{a}, \bm{\a})=\{(a_1+m_1\alpha_1 - m_2, a_2+ m_1 \a_2 -m_3, m_1) \ | \ \bf{m}=(m_1,m_2,m_3)\in \mathbb{Z}^3 \},
$$
note that the fractional part of the vector $n\bm{\a}$ lying in $\tria $ is equivalent to the lattice point inside the following set: $$B(\bf{x},N)= \tria \times (0,N]$$

Let $\chi(\bf{y})=\chi_{\bf{x},N}(\bf{y})$ be the characteristic function of the box $B(\bf{x},N)$, so the ergodic sum becomes 
$$\sum \limits_{n =1}^N \chi_{\tria}(n\bm{\alpha}\mod 1)=\sum \limits_{\bf{m} \in \mathbb{Z}^3 }\chi(a_1+m_1\alpha_1-m_2, a_2+m_1\alpha_2-m_3, m_1)
$$

Writing in matrix form:

$$
(a_1+m_1\alpha_1-m_2, a_2+m_1\alpha_2-m_3, m_1)=\left( \begin{matrix} a_1\\ a_2\\ 0 \end{matrix}\right) + \left( \begin{matrix} \a_1 &  -1  & 0 \\ \a_2 & 0 & -1 \\ 1 & 0 & 0 \end{matrix}\right) \cdot \left( \begin{matrix} m_1 \\m_2 \\m_3 \end{matrix} \right)=\bm{a}+\bf{A}\cdot \bf{m}
$$
where
$$
\bf{A}=\left( \begin{matrix} \alpha_1 &  -1 & 0 \\\a_2 & 0 & -1 \\1 & 0 & 0 \end{matrix}\right), \quad \bf{m}=\left( \begin{matrix} m_1 \\m_2 \\m_3 \end{matrix} \right), 
$$

Apply the Poisson formula to the function $f_{\bm{a}}(\bf{y})=\chi_{\bf{x}, N}(\bm{a}+\bf{A}\cdot \bf{y})$, we have 
$$\begin{aligned}
\sum \limits_{\bf{m}\in\mathbb{Z}^3}\chi(\bm{a}+\bf{A}\cdot \bf{m})
&=\sum\limits_{\bf{m}\in\Z^3} f_{\bm{a}}(\bf{m}) \\
&=\sum \limits_{\bm{\nu} \in \Z^3} \int_{\R^3} f_{\bm{a}}(\bf{y}) \cdot  e^{-2\pi \i \bm{\nu} \cdot \bf{y}} d\bf{y}   \\
&=\sum \limits_{\bm{\nu} \in \Z^3} \int_{\R^3} \chi(\bm{a}+\bf{A}\cdot \bf{y}) \cdot  e^{-2\pi \i \bm{\nu} \cdot \bf{y}} d\bf{y}   \\
&=\frac{1}{\det{\bf{A}}}\sum \limits_{\bm{\nu}\in \Z^3}\int_{\R^3} \chi(\bf{z}) e^{-2\pi \i \bm{\nu} \cdot(\bf{A}^{-1}\cdot (\bf{z}-\bm{a}))}d\bf{z}    \\
\end{aligned} 
$$

Integrate over the triangle $\tria$ for the first 2 coordinates of $\bf{z}$ and integrate over $(0,N]$ for the last coordinate $z_3$, we get: 
$$\begin{aligned}
D_{\triangle_{\bf{x}}} (\bm{a},\bm{\a}; N)
&=\sum_{\bf{m}\in \Z^3} f(\bf{m})- N\Vol(\tria) \\
&=\fc{\i^3}{(2\pi)^3} \sum_{\bf{n}\in Z^3\backslash\{0\}} 
\prod_{j=1}^2 e^{-2\pi \i n_ja_j}\(-\fc{e^{2\pi \i n_1x_1}-1}{n_1 n_2}\fc{e^{-2\pi \i (n_1\a_1+n_2\a_2-n_3)N}-1}{n_1\a_1+n_2\a_2-n_3}\)\\
&+\fc{\i^3}{(2\pi)^3} \sum_{\bf{n}\in Z^3\backslash\{0\}} 
\prod_{j=1}^2 e^{-2\pi \i n_ja_j}\(\fc{e^{2\pi\i n_1x_1}-e^{2\pi\i n_2 x_2}}{(n_1-\fc{x_2}{x_1}n_2)n_2}\fc{e^{-2\pi \i (n_1\a_1+n_2\a_2-n_3)N}-1}{n_1\a_1+n_2\a_2-n_3}\).
\end{aligned}
$$
\end{proof}

%In fact, the first part of the sum above is part of the sum appeared in (3.1) in Beck \cite{Beck}, with an additional term $\prod_{j=1}^2 e^{-2\pi \i n_ja_j}$ representing the translation by $\bm{a}$, which can be dealt with by the exact same method as in Beck \cite{Beck}, since the translation will not alter the absolute bound estimation about the ``tail'' in Section 5 of Beck \cite{Beck}, and the arguments of exponential sums in Section 6 would work exactly the same way by replacing $x_j$ with $x_j-a_j$ in the Definition (6.5) of linear form. Therefore what remains is to prove the same upper bound for the second part of the sum. 

In order to avoid technical problems with the convergence, we will not study $D_{\triangle_{\bf{x}}} (\bm{a},\bm{\a}; N)$ directly, instead, we will follow Beck\cite{Beck} and use a special weighted average of $D_{\triangle_{\bf{x}}} (\bm{a},\bm{\a}; m)$ over a $\fc{1}{N^2}$ neighborhood. To this end, we oscillate starting point $\bm{a}$ with an amplitude of $\fc{1}{N^2}$, and also the range for summation $\{1,\dots, N\}$ with amplitude of $2$, and then by using the Féjer kernal, we obtain our main object of interest. Specifically, let $\bf{u}=(u_1,u_2)$ be the oscillation of the starting point, $u_3$ be the oscillation of the summation range, then for the oscillated discrepancy:
$$
D_{\triangle_{\bf{x}}} (\bm{a}+\bf{u},\bm{\a}; u_3, N+u_3)=\sum \limits_{\mbox{\scriptsize$\begin{array}{c}
u_3< n\le N+u_3 \\\end{array}$}} \chi_{\triangle_{\bf{x}}}(\bm{a}+\bf{u}+n \bm{\a} \mod 1)- N\Vol(\triangle_{\bf{x}}),
$$
we define the $\fc{1}{N^2}$ average:
$$
\begin{aligned}
\Dbar
&=\(\fc{N^2}{2}\)^2\cdot \fc{1}{2}\int_{-\fc{2}{N^2}}^{\fc{2}{N^2}}\int_{-\fc{2}{N^2}}^{\fc{2}{N^2}}\int _{-2}^{2}\prod_{j=1}^{2}\(1-\fc{N^2}{2}|u_j|\)\cdot \l(1-\fc{|u_3|}{2}\r)\\
&\cdot D_{\triangle_{\bf{x}}} (\bm{a}+\bf{u},\bm{\a}; u_3, N+u_3)du_1du_2du_3.
\end{aligned}
$$
Using the F{\'e}jer kernel identity
$$
\fc{N^2}{2}\int_{-\fc{2}{N^2}}^{\fc{2}{N^2}} \l(1-\fc{N^2| y|}{2}\r) e^{2\pi \i ky} dy=\l(\fc{\sin(2\pi\fc{k}{N^2})}{2\pi\fc{k}{N^2}}\r)^2,
$$ 
$$
\int_{-2}^{2}\(1-\fc{|y|}{2}\)e^{2\pi \i ky} dy=\(\fc{\sin2 \pi  k }{2\pi k}\)^2.
$$
We arrive at:
\beq\label{Dbar}
\begin{aligned}
&\Dbar\\
&=\fc{\i^3}{(2\pi)^3} \sum_{\bf{n}\in \Z^3\backslash\{0\}} 
\(\fc{1-e^{2\pi \i n_2x_2}}{n_1 n_2}\(\fc{e^{-2\pi \i (n_1\a_1+n_2\a_2-n_3)N}-1}{n_1\a_1+n_2\a_2-n_3}\)\)\\
&\cdot \prod_{j=1}^2 \(e^{-2\pi \i n_ja_j} \(\fc{\sin(2\pi \fc{n_j}{N^2})}{2\pi\fc{n_j}{N^2}}\)^2\)
\cdot \(\fc{\sin 2\pi(n_1\a_1+n_2\a_2-n_3)}{2\pi (n_1\a_1+n_2\a_2-n_3)}\)^2
\\
&+\fc{\i^3}{(2\pi)^3} \sum_{\bf{n}\in \Z^3\backslash\{0\}} 
\(\fc{e^{2\pi\i n_1x_1}-e^{2\pi\i n_2 x_2}}{n_1(n_1\fc{x_1}{x_2}-n_2)}\fc{e^{-2\pi \i (n_1\a_1+n_2\a_2-n_3)N}-1}{n_1\a_1+n_2\a_2-n_3}\)\\
&\cdot \prod_{j=1}^2 \(e^{-2\pi \i n_ja_j} \(\fc{\sin(2\pi \fc{n_j}{N^2})}{2\pi\fc{n_j}{N^2}}\)^2\)
\cdot \(\fc{\sin 2\pi(n_1\a_1+n_2\a_2-n_3)}{2\pi (n_1\a_1+n_2\a_2-n_3)}\)^2\\
&=:\sum_{\bf{n} \in \Z^3\backslash \{0\}} f_1(\bf{n},\bm{a},\bf{x},\bm{\a},N)+\sum_{\bf{n} \in \Z^3\backslash \{0\}}f_2(\bf{n},\bm{a},\bf{x},\bm{\a},N)\\
&=: \Dbarone+\Dbartwo.
\end{aligned}
\eeq
We claim that the difference between $\Dbar$ and $\D$ is bounded by $C(\bm{\a},\tau,\e)(\log N)^{1+\e}$, with a constant that depends on $\bm{\a}$ and $\tau$.
\begin{Prop}{\label{Prop Dbar}}
For almost every $\bm{\a}\in [0,1)^2$ and every $\e>0$, for a fixed slope $\tau=x_1/x_2$, there exists a constant $C(\bm{\a}, \tau, \e)$, such that 
$$
|\Dbar-\D| < C(\bm{\a},\tau,\e)(\log N)^{1+\e}.
$$
\end{Prop}
\begin{proof}
We will prove that for $(u_1,u_2,u_3)\in [-\fc{2}{N^2},\fc{2}{N^2}]\times [-2,2]$, 
$$|D_{\triangle_{\bf{x}}} (\bm{a}+\bf{u},\bm{\a}; u_3, N+u_3)-\D|<C(\bm{\a}, \tau, \e)(\log N)^{1+\e}.$$
The important thing is that for given a slope $\tau$, the above constant does not depend on $u_i$ or $x_i$. 

Note that translating the starting point is equivalent to translating the the triangle, the difference above counts in fact the number of points of the translation which lie in the difference of two translated triangles. The difference between two triangles can be decomposed into three parts, the horizontal strip denoted by $H_{u;N}$, the vertical strip denoted by $V_{u;N}$, and the tilted strip denoted by $T_{u;N}$, it is sufficient to prove that each strip does not contain more points than the desired bound.

For $H_{u;N}$ and $V_{u;N}$, the number of points of the sequence $\{n \bm{\a} \mod 1 \ | \ 1\le n\le N \}$ inside the strips is limited by the number of $n$'s such that $\|n \a_1\|\le 2/N^2$ or $\|n\a_2\|\le 2/N^2$. Since the series $\sum_{n=1}^{\infty} \fc{1}{n^2}$ is convergent, by a standard application of Borel-Cantelli Lemma, the number of $n$'s that solve the inequality $\|n\a_i\|< 2/n^2$ is finite for almost every $\a_i\in[0,1)$, therefore there exists a constant $C(\bm{\a})$ such that:
$$
\#\{1\le n\le N \ | \ 
\|n\a_i\|< 2/N^2, \ i=1 \text{ or } 2
\}<C(\bm{\a}),
$$
thus the number of points inside $H_{u;N}$ and $V_{u;N}$ has an upper bound $C(\bm{\a})$. 

As for the tilted strip $T_{u;N}$, consider the vectors formed by any two different points inside $T_{u;N}$, then the vectors form a subset of the sequence $\{n\bma\mod 1\}_{-N\le n\le N}$, and they fall inside the the strip $T_{\tau; N}$ for any $|\bf{u}|\le 2/N^2$, where $T_{\tau;N}$ is a parallelogram centered at the origin, with the slope $\tau$, and of volume $O(1/N^2)$. Since for any $n$, $\int_{\T^2} \chi_{T_{\tau;N}}( n\bm{\a} \mod 1) d\bm{\a} =\Vol (T_{\tau;N})=O(1/N^2)$, we have
$$
\bal
&\text{mes}\l\{\bm{\a}\in \T^2 \ \Big| \ \sum_{n=-N}^{N} \chi_{T_{\tau;N}}( n\bm{\a} \mod 1)\ge (\log N)^{1+\e}\r\}\\
&\le \fc{1}{(\log N)^{1+\e}} \sum_{n=-N}^{N} \Vol(T_{\tau;N})\\
&=O\l(\fc{1}{(\log N)^{1+\e} N}\r),\\
\eal
$$
which summable for $N\ge 1$, therefore by Borel-Cantelli Lemma, it happens only finitely many times that the desired bound $\#\{-N\le n\le N \ |\  (n\bma\mod 1)\in T_{\tau;N}\}< (\log N)^{1+\e}$ is violated for almost every $\bma\in [0,1]^2$ , which means that there exists a constant $C(\bm{\a}, \tau, \e)$, such that:
$$
\#\{-N\le n\le N \ |\  n\bm{\a} \mod 1 \in T_{\tau,N} \} \le C(\bm{\a}, \tau, \e)(\log N)^{1+\e}.
$$
Since the vectors are no fewer than the points inside $T_{u;N}$ for any $\bf{u}$, we have the desired bound for all strips $T_{u;N}$ with a given slope $\tau$.  
%Let $\bm{\a}$ be a random variable uniformly disributed on $\T^2$, and $A(n,\tau;N)$ denote the event that $n\bm{\a} \mod 1$ lie inside $T_{\tau;N}$, since the volume of the strip is $O(1/N^2)$, $\mathbb{P}(A(n,\tau;N)=O(1/N^2)$, and 
\end{proof}

By Proposition \ref{Prop Dbar}, we can now shift our attention to the asymptotic behavior of $\Dbar$, which has better convergence property as shown in Section 4.

\section{Local lemmas}
%\textcolor{red}{This section needs to be adapted to the new denominator $n_1(n_1\tau-n_2)(n_1\a_1+n_2\a_2+n_3)$}
In this section we introduce four lemmas that describes some ``almost always'' properties of $\bm{\a}\in \R^k$, these are modified versions of the lemmas in Section 4 of Beck\cite{Beck}, and can be proved in the exact same way by modifying the denominators in the original proof of Beck.

The first lemma estimates a sum of the ``large terms'' in $\Dbar$ \eqref{Dbar}.
\begin{lem}\label{lemma for sum of small divisors}
Let $\varphi(n)$ be an arbitray positive increasing function of $n$ with $\sum_{n=1}^{\infty} \fc{1}{\varphi(n)}< \infty$. Then for almost every $\bm{\a}\in \R^2$, with $\bf{n}\bm{\a}=n_1\a_1+n_2\a_2$, and every $\tau\in [0,1]$ irrational, 
\beq\label{Lemma for large terms}
\sum_{\bf{n}\in U(\bma, \tau;N)} \l(\max\{1,|n_1|\}|n_1\tau-n_2|\cdot \|\bf{n}\bm{\a}\|\r)^{-1} \ll (\log N)^k\cdot \varphi(\log \log N)
\eeq
holds for all $N$, where
\beq\label{Definition for set of large terms}
U(\bma,\tau;N)=\l\{
\bf{n}\in \Z^2\backslash \{\bf{0}\} \ \l| \
\bal & |n_j|\le N^2\cdot (\log N)^2, 1\le j\le 2; \\
&\max\{1,|n_j|\}|n_1\tau-n_2|\cdot \|\bf{n}\bm{\a}\|\le (\log N)^{40}\\
&|n_1\tau-n_2|\ge 1/2
\eal \r.
\r\}
\eeq
\end{lem}
\begin{proof}
This is a modified version of Lemma 4.1 of Beck\cite{Beck}, the key difference here is that $n_2$ is replaced by $n_1\tau-n_2$, but since the additional condition $|n_1\tau-n_2|\ge 1/2$ helps avoid the difficult situation when $\|n_1\tau\|$ is small, Beck's proof works here in the exact same way after replacing $n_2$ by $n_1\tau-n_2$. 
\end{proof}
The following lemma is used to handle the ``small terms'' in $\Dbar$.

Let $\bf{V}=(v_1,\dots,v_k)\in \Z^k$ and $\bf{W}=(w_1,\dots, w_2)\in \Z^2$ satisfy $v_j<wj$, $1\le j\le 2$, that is , $\bf{V}\le \bf{W}$. Let $Q(\bf{V},\bf{W})$ denote the lattice box
\beq\label{Lattic box}
Q(\tau; \V,\W)=\l\{\bf{n}\in \Z^2 \l| \ \bal & v_1\le n_1<w_1, \\ 
& v_2\le n_1\tau-n_2<w_2 \\
& |n_1\tau-n_2|\ge 1/2
\eal
\r. \r\}.
\eeq
For every real vector $\bm{\a}\in \R^2$, constant $C\ge 1$, integral vectors $\V\in \Z^2$ and $\W\in\Z^2$ with $\V<\W$, let $Z(\bm{\a},\tau; C; \V,\W)$ denote the number of integral solutions $\bf{n}\in Q(\tau; \V,\W)$ of inequality below for the fractioanl part:
\beq\label{fractional inequality with small denominators}
0<\{\n\bm{\a}\}<\min\l\{ \fc{1}{2},\fc{C}{\max\{1,|n_1|\}\cdot |n_1\tau-n_2|}\r\}.
\eeq
The ``expected value'' of $Z(\bm{\a},\tau; C; \V,\W)$, $\bma\in\R^2$, is
\beq\label{Expected value for solutions of fractional inequality}
E(\tau; C;\V,\W)=\sum_{\n\in Q(\tau; \V,\W)} \min\l\{ \fc{1}{2},\fc{C}{\max\{1,|n_j|\}\cdot |n_1\tau-n_2|}\r\}.
\eeq
\begin{lem}\label{Lemma of second moment for small terms}
For almost every $\bma\in \R^k$ ($k\ge 2$), and every $\tau\in [0,1]$ irrational, 
$$
Z(\bma,\tau; C; \V,\W)=E(\tau, C;\V,\W)+O(C^{\fc{3}{4}+\e}\cdot(\log N)^{(\fc{1}{2}+\e)k} )
$$
holds for all $C=q^4$, $q=1,2,3,\dots$, and for all lattice boxes $Q(\tau; \V,\W)$ with $-N\le \V<\W\le N$, $N=2,3,4\dots$. 
\end{lem}
It is evident that Lemma \ref{Lemma of second moment for small terms} remains true if we replace \eqref{fractional inequality with small denominators} by its complement:
$$
1-\min\l\{ \fc{1}{2},\fc{C}{\max\{1,|n_j|\}\cdot |n_1\tau-n_2|}\r\}<\{\n\bm{\a}\}<1.
$$
\begin{proof}
This is a modified version of Lemma 4.2 of Beck\cite{Beck}, but similar to Lemma 3.1, the condition $|n_1\tau-n_2|\ge 1/2$ helps avoid the difficult situation when $\|n_1\tau\|$ is small, and Beck's proof works here in the exact same way after replacing $n_2$ by $n_1\tau-n_2$. 
\end{proof}

The next lemma is special case of Khintchine's local criterion for linear forms.
\begin{lem}\label{Lemma for divergent law for linear forms}
Let $\varphi(n)$ be an arbitrary positive increasing function of $n$ with $\sum_{n=1}^{\infty}\fc{1}{\varphi(n)}=\infty$. Then for almost every $\bma\in \R^k$, there are infinitely many strictly positive integeral vectors $\n=(n_1,\dots, n_k)\in\N^k$ such that 
\beq\label{divergent law for linear forms}
n_1\dots n_k\cdot (\log (n_1\dots n_k))^k\cdot \varphi(\log \log (n_1^2\dots n_k^2)) \cdot \|\n\bma\|<1.
\eeq
\end{lem}
The last lemma helps control the divisors when $|n_1\tau-n_2|=\|n_1\tau\|$ is small and $|\n|$ is big:
\begin{lem}\label{Lemma for large n_1 and n_1* tau-n_2}
Suppose that $\tau$ is irrational, then for almost every $\bma\in\R^2$, and almost every $\tau\in[0,1]$, the series
$$
\sum_{\n\in\Z^2\backslash \{\bf{0}\}} \l(\|\n\bma\|\cdot (\max\{1,|\log |n_1|,\log |n_2|\})^{4} \cdot \max\{1,|n_1|\}\cdot |n_1\tau-n_2|\r)^{-1}
$$
converges.
%$$
%\sum_{\n\in\Z^k\backslash \{\bf{0}\}} \l(\|\n\bma\|\cdot (\max\{1,|\log |n_1 \cdots n_k|\})^{k+2} \cdot \prod_{j=1}^{k}\max\{1,|n_j|\}\r)^{-1}
%$$
%converges.
\end{lem}
\begin{proof}
We denote the sum above by $S$, $S$ can be divided into 2 parts: $S=S_1+S_2$, where
$$
S_1=\sum_{\mbox{\scriptsize $\ba{c} n_1\neq 0\\ n_2: |n_1\tau-n_2| =\|n_1\tau\| \ea$}} \l(\|\n\bma\|\cdot (\max\{1,|\log |n_1|,\log |n_2|\})^{4} \cdot \max\{1,|n_1|\}\cdot |n_1\tau-n_2|\r)^{-1},
$$
and 
$$
S_2=\sum_{\n: |n_1\tau-n_2|\ge 1/2} \l(\|\n\bma\|\cdot (\max\{1,|\log |n_1|,\log |n_2|\})^{4} \cdot \max\{1,|n_1|\}\cdot |n_1\tau-n_2|\r)^{-1}.
$$
Let $\d>0$, and consider the integral with $\n\in\Z^2\backslash\{0\}$.
$$
J_1(n)=\int_0^1\int_0^1\int_0^1 \l(\|\n\bma\|\l|\log \|\n\bma\|\r|^{1+\d} \|n_1\tau\| |\log \|n_1\tau\||^{1+\d}\r)^{-1} d\a_1d\a_2 d\tau.
$$
The integral has an finite value independent of $\n$. 
Hence the sum 
$$
\sum_{\mbox{\scriptsize $\ba{c} n_1\neq 0\\ n_2: |n_1\tau-n_2| =\|n_1\tau\| \ea$}} \fc{J_1(\n)}{\max\{1,|n_1|\cdot (\log |n_1|)^{1+\d}\}}
$$
is convergent, and so 
\beq\label{convergence for S1}
\sum_{\mbox{\scriptsize $\ba{c} n_1\neq 0\\ n_2: |n_1\tau-n_2| =\|n_1\tau\| \ea$}} \l(\|\n\bma\|\l|\log \|\n\bma\|\r|^{1+\d} \|n_1\tau\| |\log \|n_1\tau\||^{1+\d}\max\{1,|n_1|\cdot (\log |n_1|)^{1+\d}\r)^{-1}
\eeq
is convergent for almost every $\bma\in \R^2$, and almost every $\tau\in[0,1]$. 
Since $\sum_{n=1}^{\infty} n^{-2}=O(1)$, the inequality 
$$
\|\n \bma \|<\l(\prod_{j=1}^{2} \max\{1,|n_j|\}\r)^{-2}
$$
has only a finite number of integral solutions $\n\in\Z^2$ for almost every $\bma\in\R^2$.
Hence 
\beq\label{Inequality for log (n*alpha)< log n}
|\log \|\n\bma\||< 2\sum_{j=1}^2 \max\{0,\log |n_j|\}\ll \max\{0,\log |n_1|\},
\eeq
the last part of the inequality comes from $|n_1\tau-n_2|\le 1/2$. 
By the same reasoning we have:
\beq\label{Inequality for log (n_1*tau)< log n_1}
|\log \|n_1\tau\||< 2 \max\{0,\log |n_1|\}
\eeq
By \eqref{convergence for S1}, \eqref{Inequality for log (n*alpha)< log n}, and \eqref{Inequality for log (n_1*tau)< log n_1}, we have the series:
$$
\bal
&S_1 \ll \\ 
& \sum_{\mbox{\scriptsize $\ba{c} n_1\neq 0\\ n_2: |n_1\tau-n_2| =\|n_1\tau\| \ea$}} \l(\|\n\bma\|\l|\log \|\n\bma\|\r|^{1+\d} \|n_1\tau\| |\log \|n_1\tau\||^{1+\d}\max\{1,|n_1|\cdot (\log |n_1|)^{1+\d}\r)^{-1},\\
\eal$$
which is convergent by taking $\d=1/3$.

For $S_2$, we employ the same method, the integral
$$
J_2(n)=\int_0^1\int_0^1 \l(\|\n\bma\|\l|\log \|\n\bma\|\r|^{1+\d}\r)^{-1} d\a_1d\a_2 .
$$
The integral has an finite value independent of $\n$. 
Hence the sum 
$$
\sum_{\mbox{\scriptsize $\ba{c} \n: |n_1\tau-n_2| \ge 1/2 \ea$}} \fc{J_2(\n)}{n_1\max\{1,(\log |n_1|)^{1+\d}\}|n_1\tau-n_2|\max\{1,(\log |n_1\tau-n_2|)^{1+\d}\}}
$$
is convergent, and so 
$$
\bal
\sum_{\n:|n_1\tau-n_2|\ge 1/2} &\fc{1}{\|\n\bma\|\l|\log \|\n\bma\|\r|^{1+\d}}\cdot \fc{1}{n_1\max\{1,(\log |n_1|)^{1+\d}\}}\\
&\cdot\fc{1}{|n_1\tau-n_2|\max\{1,(\log |n_1\tau-n_2|)^{1+\d}\}}\\
\eal
$$
is convergent for almost every $\bma$ and every $\tau$ irrational. 
By \eqref{Inequality for log (n*alpha)< log n} and $|n_1\tau-n_2|\le |n_1|+|n_2|\le 2\max\{|n_1|,|n_2|\}$, the series $S_2$ is convergent by taking $\d= 1/3$. 
\end{proof}
\section{Estimating the "tail" of the discrepancy function}
Recall that $\Dbar$ is the sum of $\Dbarone$ and $\Dbartwo$(see \eqref{Dbar}), where $\Dbarone$ is already dealt with in Beck\cite{Beck}, see Lemma 4.1 below, this section is devoted to estimating the ``tail" of the second part of sum $\Dbartwo$.
\begin{lem}
Let $\varphi(n)$ be an arbitrary positive increasing function of $n$ with $\sum_{n=1}^{\infty}\fc{1}{\varphi(n)}<\infty$, then for almost every $\bma\in \T^2$,
\beq\label{estimation for Dbarone}
\max_{\bf{x},\bm{a} \in \T^2}|\Dbarone| \ll (\log N)^2\varphi(\log\log N).
\eeq
\end{lem}
\begin{proof}
Note that $\Dbarone$ in \eqref{Dbar} is in fact part of the sum (3.12) in Beck\cite{Beck}, whose estimation is given by (6.1) for $\sum_9$ and (8.14) for $\sum_{10}$ in Beck\cite{Beck}. Our sum $\Dbarone$ has an additional factor of $\prod_{j=1}^{2}e^{-2\pi \i n_j a_j}$, which can be estimated in the same fashion, or it could be understood as part of the discrepancy function for the translated boxes $\prod_{i=1}^{2} [a_i,a_i+x_i]$, which is a finite sum or difference of the orginal boxes of the form $\prod_{i=1}^{2} [0,x_i]$, therefore satisfies the same estimation as the original boxes.
\end{proof}
Note that $\Dbartwo$ is the sum of the products (where $\bf{n}\in \Z^3 \backslash {\bf{0}}$)
\beq{\label{term}}
\begin{aligned}\ftwo=
&\fc{\i^3}{(2\pi)^3} \sum_{\bf{n}\in \Z^3\backslash\{0\}} 
\(\fc{e^{2\pi\i n_1x_1}-e^{2\pi\i n_2 x_2}}{n_1(n_1\fc{x_1}{x_2}-n_2)}\fc{e^{-2\pi \i (n_1\a_1+n_2\a_2-n_3)N}-1}{n_1\a_1+n_2\a_2-n_3}\)\\
&\cdot \prod_{j=1}^2 \(e^{-2\pi \i n_ja_j} \(\fc{\sin(2\pi \fc{n_j}{N^2})}{2\pi\fc{n_j}{N^2}}\)^2\)
\cdot \(\fc{\sin 2\pi(n_1\a_1+n_2\a_2-n_3)}{2\pi (n_1\a_1+n_2\a_2-n_3)}\)^2\\
\end{aligned}
\eeq
let:
$$
\bar{D}_5(\bf{x},\bm{a},\bm{\a};N)=\sum_{\bf{n}\in U_5(\bm{\a},\tau)}    \ftwo,
$$
and 
\beq
U_5(\n,\tau;N)=\l\{\n \in \Z^3\backslash \{\bf{0}\} \l| 
\bal &\max\{n_1,|n_1\tau-n_2|\} \le N^2/4, \\
&\min\{n_1,|n_1\tau-n_2|\}\ge(\log N)^{40}, \\
& |\n\bma-n_3|=\|\n\bma \|,\ |n_1\tau-n_2|\ge 1/2,\\
& |n_1|n_1\tau-n_2|\|\bf{n}\bma\|>(\log N)^{40}. \\
\eal\r.
\r\}
\eeq

The main result of this section is the following:

\begin{Prop}\label{D7}
Let $\varphi(n)$ be an arbitrary positive increasing function of $n$ with $\sum_{n=1}^{\infty}\fc{1}{\varphi(n)}< \infty$, then for almost every $\bma\in[0,1]^2$, and almost every $\tau \in [0,1]$, we have 
$$|\bar{D}_5(\bf{x},\bm{a},\bm{\a};N)-\Dbartwo|\ll (\log N)^2\varphi^2(\log \log N)$$
\end{Prop}

To prove Proposition $\ref{D7}$, we need to control different components of $\bar{D}_5-\bar{D}_{\triangle_{\bf{x}},2}$ step by step.

\subsection{Estimation for the sum when $|n_1|$ or $|n_2|$ is big}

%First we want to show that the sum of the terms when $|n_1|$ is big does not contribute much.
Define
\beq
\bar{D}_1(\bf{x},\bm{a},\bm{\a};N)=\sum_{\bf{n}\in U_1(\bm{\a};N)}   \ftwo,
\eeq
where 
$$
U_1(\bm{\a};N)=\l\{\bf{n}\in \Z^3\backslash \{\bf{0}\} \ |\ 
 \max\{|n_1|,|n_2|\}\le N^2(\log N)^2
\r\}
.$$
We define another set that is parallel to $U_{\bma}$ and abuse the notation $\n$ for $(n_1,n_2)$ if $n_3$ is the closest integer to $n_1\a_1+n_2\a_2$:
$$
U'_1(\bm{\a};N)=\l\{\bf{n}\in \Z^2\backslash \{\bf{0}\} \ |\ 
 \max\{|n_1|,|n_2|\}\le N^2(\log N)^2
\r\}
.$$
We show the following:
\begin{Prop}\label{D_1 - D}
For almost every $\bma\in \R^2$, and almost every $\tau\in [0,1]$, we have
$$\l|\bar{D}_1(\bf{x},\bm{a},\bm{\a};N)-\Dbartwo \r|=O(1),$$
where $O(1)$ represents an absolute bound which depends on $\bma, \ \tau$, but does not depend on $\bm{a}$ or $N$.
\end{Prop}
\begin{proof}
First we decompose the difference into 2 parts:
\beq{\label{D1=D11+D22}}
\l|\bar{D}_1(\bf{x},\bm{a},\bm{\a};N)-\Dbartwo\r|\ll\bar{D}_{1,1} +\bar{D}_{1,2},
\eeq
where 
$$
\bar{D}_{1,1}=\sum_{\mbox{\scriptsize $\ba{c} \n\notin U'_1(\bm{\a})  \\
 \ea$}}
  \l( \|\n\bma\| |n_1|  |n_1\tau-n_2| \prod_{j=1}(\max\{1,\fc{n_j}{N^2}\})^2 \r)^{-1}
 $$
%$$
%\bar{D}_{1,1}=\sum_{\mbox{\scriptsize $\ba{c} |n_1|>N^2(\log N)^2 \\
%|n_1\tau-n_2|=\|n_1\tau\|
% \ea$}} \l(\|\bf{n}\bma\| |n_1| \l(\fc{n_1}{N^2}\r)^2 \|n_1\tau\| \r)^{-1},
%$$
%$$
%\bar{D}_{1,2}=\sum_{\mbox{\scriptsize $\ba{c} |n_1|>N^2(\log N)^2 \\
%1/2\le |n_1\tau-n_2|\le N^2(\log N)^2
% \ea$}}
%  \l(\|\bf{n}\bma\| |n_1| \l(\fc{n_1}{N^2}\r)^2 |n_1\tau-n_2|\r)^{-1},
%$$
%%and
%$$
%\bar{D}_{1,3}=\sum_{\mbox{\scriptsize $\ba{c} |n_1|>N^2(\log N)^2 \\
% |n_1\tau-n_2|> N^2(\log N)^2
% \ea$}}
%  \l(\|\bf{n}\bma\| |n_1| \l(\fc{n_1}{N^2}\r)^2 |n_1\tau-n_2|(\max\{1,\fc{n_2}{N^2}\})^2 \r)^{-1},
%$$
%and finally,
$$
\bar{D}_{1,2}=\sum_{r=1}^{\infty}\sum_{ \mbox{\scriptsize $\ba{c} \n\notin U'_1(\bm{\a})
 \ea$}}
  \l( |n_1|  |n_1\tau-n_2| \prod_{j=1}(\max\{1,\fc{n_j}{N^2}\})^2 \r)^{-1}\cdot \fc{1}{r^3},
$$
where $r=\lceil |n_1\a_1+n_2\a_2-n_3|\rceil$.

To estimate $\bar{D}_{1,1}$, we employ Lemma \ref{Lemma for large n_1 and n_1* tau-n_2}. For every $\bma\in\R^2$ and every $\tau\in [0,1]$ that satisfy Lemma \ref{Lemma for large n_1 and n_1* tau-n_2},
$$
\bal
\bar{D}_{1,1} 
&\ll\sum_{\n\notin U'_1(\bm{\a})}  \l(\|\n\bma\|\cdot (\max\{1,|\log |n_1|,\log |n_2|\})^{4} \cdot \max\{1,|n_1|\}\cdot |n_1\tau-n_2|\r)^{-1}\\
&
\le \sum_{\n\in\Z^2\backslash \{\bf{0}\}} \l(\|\n\bma\|\cdot (\max\{1,|\log |n_1|,\log |n_2|\})^{4} \cdot \max\{1,|n_1|\}\cdot |n_1\tau-n_2|\r)^{-1}\\
&=O(1).
\eal
$$
$\bar{D}_{1,2}$ can be easily controlled by $\bar{D}_{1,1}$:
$$
\bal
\bar{D}_{1,2} 
&\ll \l(\sum_{r=1}^{\infty} \fc{1}{r^3}\r)\l(\sum_{ \mbox{\scriptsize $\ba{c} \n\notin U'_1(\bm{\a})
 \ea$}}  \l( |n_1|  |n_1\tau-n_2| \prod_{j=1}(\max\{1,\fc{n_j}{N^2}\})^2 \r)^{-1}\r) \\
& \ll \bar{D}_{1,1} \\
& =O(1).
\eal
$$
\end{proof}

From the proof it is easy to see that the series $\Dbartwo$ is absolutely convergent for almost every $\a\in \R^2$ and $\tau\in[0,1]$.

\subsection{Estimation for the sum when $|\n \bma-n_3|$ is bigger than $\fc{1}{3}$.}
Let 
$$
\bar{D}_2(\bf{x},\bm{a},\bm{\a};N)=\sum_{\bf{n}\in U_2(\bma;N)}  \ftwo
$$
where $\bf{n}=(n_1,n_2,n_3)$ satisfies
\beq{\label{bigger than one third}}
U_2(\bma;N)=
\l\{\n\in\Z^3 \ | \max\{|n_1|,|n_2|\}\le N^2(\log N)^2,\quad
|\n \bma -n_3|< \fc{1}{3}
\r\}
\eeq
We show that the difference between $\bar{D}_1$ and $\bar{D}_2$ can be well controlled, hence we only need to focus on the sum over $(n_1, n_2)$ instead of $(n_1,n_2,n_3)$:
\begin{Prop}
Let $\varphi(n)$ be an arbitrary positive increasing function of $n$ with $\sum_{n=1}^{\infty}\fc{1}{\varphi(n)}<\infty$, then for almost every $\bma \in \R^2$ and almost every $\tau\in [0,1]$, we have
$$
\l|\bar{D}_2(\bf{x},\bm{a},\bm{\a};N)-\bar{D}_1(\bf{x},\bm{a},\bm{\a};N)\r| \ll (\log N)^2\varphi(\log\log  N),
$$
\end{Prop}
\begin{proof}
Similar to the previous proposition, the difference above can be decomposed into 2 parts:
\beq\label{D2=D21+D22}
\bal
&\l|\bar{D}_1(\bf{x},\bm{a},\bm{\a};N)-\bar{D}_2(\bf{x},\bm{a},\bm{\a};N)\r|\\
&\ll \sum_{r=1}^{\infty}\sum_{\n\in U'_1(\bma)}
\l(|n_1|  |n_1\tau-n_2| \prod_{j=1}(\max\{1,\fc{n_j}{N^2}\})^2 \r)^{-1}\cdot \fc{1}{r^3}\\
&\ll\bar{D}_{2,1} +\bar{D}_{2,2},\\
\eal
\eeq
where 
$$
\Dtwoone=\sum_{1\le |n_1|\le N^2(\log N)^2} \fc{1}{|n_1| \| n_1\tau\|},
$$
and 
$$
\Dtwotwo=\sum_{\n\in U'_1(\bma): |n_1\tau-n_2|\ge 1/2} \fc{1}{|n_1||n_1\tau-n_2|}.
$$
$\Dtwotwo$ can be estimated easily:
$$
\Dtwotwo\ll \prod_{j=1}^{2}\l(\sum_{1\le |n_j|\le N^2(\log N)^2} \fc{1}{|n_j|}\r)
\ll (\log N)^2.
$$
For $\Dtwoone$, we apply the continued fraction algorithm as in the proof for (9.1) in Beck\cite{Beck}, see \cite{Lang} for the theory of continued fractions. Let $a_j$ be the j-th partial quotient of $\tau$ and $q_j$ be the dnominator of the j-th convergent, then by Dirichlet's ``box principle'' and the estimation for the j-th convergent: 
$$
\frac{1}{2q_{j+1}}\le \|q_j\tau\|\le \fc{1}{q_{j+1}},
$$
we have:
$$
\sum_{n_1=q_j}^{q_{j+1}} \frac{1}{n_1\|n_1\tau\|}\ll \frac{q_{j+1}}{q_j} \sum_{k=1}^{q_{j+1}-q_{j}-1} \frac{1}{k}\ll a_{j+1}\cdot \log (q_{j+1}).
$$
Let $s$ be the integer such that $q_s\le N^{2+\e}<q_{s+1}$, then for almost every $\tau$, we have 
$$
s\ll \log N.
$$
combined with $\log q_j\ll j$ for almost every $\tau$, we have 
$$
\sum_{n_1=1}^{N^2(\log N)^2} \fc{1}{n_1\|n_1\tau\|}\ll \l(\sum_{j=1}^{s} a_{j+1} \r) \log N.
$$
By a theorem of Khintchine about the the sum of partial quotient (see \cite{Koksma}, p.46), if $\psi(n)$ is any increasing function, then
$$
a_1(\tau)+\dots +a_s(\tau) \ll s\cdot \psi(s) \Longleftrightarrow \sum_{n=1}^{\infty} \fc{1}{n\cdot \psi(n)}<\infty
$$
for almost every $\tau$. Let $\psi(n)=\varphi(\log n)$, then 
$$
\sum_{n_1=1}^{N^2(\log N)^2} \fc{1}{n_1\|n_1\tau\|}\ll \l(\sum_{j=1}^{s} a_{j+1} \r) \log N\ll (\log N)^2\varphi(\log \log N).
$$
\end{proof}
\subsection{Estimation for the sum when $|n_1 \tau-n_2|=\|n_1\tau\|$.}
An important difference between the case of triangles and the case of boxes is that the divisor changes from $n_1 n_2(n_1\a_1 +n_2\a_2 -n_3)$ to $n_1(n_1\tau-n_2)(n_1\a_1+n_2\a_2-n_3)$, which adds a possibility for $|n_1\tau-n_2|$ to be small. In this section, we deal with the difference by some metrical properties of  $n_1\|n_1\a_1+n_2\a_2\|$, where $n_2$ is the closest integer to $n_1\tau$. Unfortunately, this would lead to a loss of $\varphi(\log \log N)$ in the estimation.

Let 
$$
\bar{D}_3(\bf{x},\bm{a},\bm{\a};N)=\sum_{\bf{n}\in U_3(\alpha,\tau)}  \ftwo
$$
where $\bf{n}=(n_1,n_2,n_3)$ satisfies
\beq{\label{n1*tau-n2 bigger than one third}}
U_3(\bma,\tau;N)=
\l\{\n\in\Z^3 \ \l| \ 
\bal 
&\max\{|n_1|,|n_2|\}\le N^2(\log N)^2,\\
&|\n \bma -n_3|< \fc{1}{3}, \\
&|n_1\tau-n_2|=\|n_1\tau\|\\
\eal
\r.\r\}
\eeq
We show that the difference between $\bar{D}_3$ and $\bar{D}_2$ can be controlled by the following proposition, hence we only need to focus on the sum over $(n_1, n_2)$ when $|n_1\tau-n_2|\ge 1/2$, which could be treated in the same way as Beck did.
\begin{Prop}
Let $\varphi(n)$ be an arbitrary positive increasing function of $n$ with $\sum_{n=1}^{\infty}\fc{1}{\varphi(n)}<\infty$, then for almost every $\bma \in \R^2$ and almost every $\tau\in [0,1]$, we have
$$
\l|\bar{D}_3(\bf{x},\bm{a},\bm{\a};N)-\bar{D}_2(\bf{x},\bm{a},\bm{\a};N)\r| \ll (\log N)^2\varphi^2(\log \log N),
$$
\end{Prop}
\begin{proof}
Note that $\tau=x_1/x_2$, and $x_2\in (0,1]$, we use the inequality:
$$\l|\fc{e^{2\pi\i n_1x_1}-e^{n_2x_2}}{n_1\tau-n_2}\r|= \l|x_2\fc{e^{2\pi\i n_1x_1}-e^{2\pi \i n_2x_2}}{n_1x_1-n_2x_2} \r|\le1,$$
therefore $\l|\ftwo\r|\ll \fc{1}{|n_1|\|\n\bma\|}$.

Let $n^*_2$ denote the closest integer to $n_1\tau$, then
$$
|\bar{D}_3(\bf{x},\bm{a},\bm{\a};N)-\bar{D}_2(\bf{x},\bm{a},\bm{\a};N)|\ll \sum_{1\le |n_1|\le N^2(\log N)^2} \fc{1}{|n_1| \|n_1\a_1 + n_2^* \a_2\|}
$$ 
we employ the same method as in Lemma \ref{Lemma for large n_1 and n_1* tau-n_2}, the integral
$$
J_2(\n)=\int_0^1\int_0^1 \l(\|\n\bma\|\l|\log \|\n\bma\|\r| \varphi(\log \log \| \n\bma\|)\r)^{-1} d\a_1d\a_2 .
$$
has an finite value independent of $\n$. 
Hence the sum 
$$
\sum_{\mbox{\scriptsize $\ba{c}  |n_1|=1 \ea$}}^{\infty} \fc{J_2(n_1, n^*_2)}{|n_1|\log |n_1| \varphi(\log \log|n_1|)}
$$
is convergent, and so for almost every $\bma$, the series 
$$
\bal
S_3=\sum_{|n_1|=1}^{\infty} \l(|n_1|\log |n_1| \varphi(\log \log|n_1|)\|\n\bma\|\l|\log \|\n\bma\|\r| \varphi(\log \log \| \n\bma\|) \r)^{-1}\\
\eal
$$
is convergent, here $\n\bma$=$n_1\a_1+n^*_2\a_2$.
Combined with the inequality \eqref{Inequality for log (n*alpha)< log n}, we have for almost every $\bma\in \R^2$
$$
\bal
&\sum_{1\le |n_1|\le N^2(\log N)^2} \fc{1}{|n_1| \|n_1\a_1 + n_2^* \a_2\|}\\
&\ll \sum_{|n_1|=1}^{N^2(\log N)^2} \fc{(\log N)^2\varphi^2(\log \log N)}{|n_1|\log |n_1| \varphi(\log \log|n_1|)\|\n\bma\|\l|\log \|\n\bma\|\r| \varphi(\log \log \| \n\bma\|)}\\
&\ll(\log N)^2\varphi^2(\log \log N) \cdot S_3\\ 
&\ll (\log N)^2\varphi^2(\log \log N) 
\eal
$$
\end{proof}

\subsection{Estimation for the sum when $|n_1||n_1 \tau-n_2 \|\|\n \bma\|$ is small.}
From now on, we can restrcit the range of our discussion to $\bf{n}=(n_1,n_2)$, Define:
\beq\label{sum of small divisors}
\bar{D}_4(\bf{x},\bm{a},\bm{\a};N)=\sum_{\n\in U_4(\n,\tau;N)}  \ftwo
\eeq
where %\eqref{condition for n_2 n_3} and \eqref{condition for small divisors} are defined below respectively. 
%(\eqref{condition for n_2 n_3} means that $n_2$ and $n_3$ are respectively the closest integer to $n_1 \a$ and $n_1\b$.
%\eqref{condition for small divisors} means $\{n_1\in \Z: 1\le |n_1| \le N^2(\log N)^2, n_1 \|n_1 \a \| \|n_1\b \|< (\log N)^{40} \}$.)

%With the previous step 1 and 2, we only need to focus on the sum of the product \ftwo where $n_1< N^2 (\log N)^2 $, and 

\beq{\label{condition for small divisors}}
U_4(\n,\tau;N)=\l\{\n \in \Z^3\backslash \{\bf{0}\} \l| 
\bal &\max\{n_1,n_2\} \le N^2(\log N)^2, \\
& |\n\bma-n_3|=\|\n\bma \|,\ |n_1\tau-n_2|\ge 1/2,\\
& |n_1|n_1\tau-n_2|\|\bf{n}\bma\|< (\log N)^{40} \\
\eal\r.
\r\}
\eeq 
From Lemma \ref{lemma for sum of small divisors}, we have the following estimation: 
\begin{Prop}\label{control for sum of small divisors}
Let $\varphi(n)$ be an arbitrary positive increasing function of $n$ with $\sum_{n=1}^{\infty} \fc{1}{\varphi(n)}< \infty$, then for almost every $\bma\in\R^2$ and every $\tau\in [0,1]$ irrational, we have
$$
\l|\bar{D}_4(\bf{x},\bm{a},\bm{\a};N)-\bar{D}_3(\bf{x},\bm{a},\bm{\a};N) \r|\ll (\log N)^2\varphi(\log \log N).
$$
where the constant may depend on $\bma$ and $\tau$.
\end{Prop}
%HERE WE STILL NEED SOME IMPROVEMENTS in order to get power $2+\e$!!

\subsection{Estimation for the sum when $n_1$ or $n_2$ is between $N^2/4$ and $N^2(\log N)^2$.}
Define:
$$
\bar{D}_5(\bf{x},\bm{a},\bm{\a};N)=\sum_{\n\in U_5(\n,\tau;N)}  \ \ftwo
$$
where we recall that 
$$
U_5(\n,\tau;N)=\l\{\n \in \Z^3\backslash \{\bf{0}\} \l| 
\bal &\max\{n_1,|n_1\tau-n_2|\} \le N^2/4, \\
&\min\{n_1,|n_1\tau-n_2|\}\ge(\log N)^{40}, \\
& |\n\bma-n_3|=\|\n\bma \|,\ |n_1\tau-n_2|\ge 1/2,\\
& |n_1|n_1\tau-n_2|\|\bf{n}\bma\|>(\log N)^{40}. \\
\eal\r.
\r\}
$$

The goal of this step is to prove the following:
\begin{Prop}\label{sum7}
For almost every $\bma\in \R^2$, we have 
$$\l|\bar{D}_5(\bf{x},\bm{a},\bm{\a};N)-\bar{D}_4(\bf{x},\bm{a},\bm{\a};N)\r|\ll(\log N)^2 \log \log N
$$
\end{Prop}
By limiting the range of $(2\pi n_1/N^2)$ in $\ftwo$ to $(-\pi/2, \pi/2)$, $f_2$ becomes better-behaved for later estimations.
\begin{proof}
For every $\bf{v}\in\Z^2$, with $1\le 2^{v_j}\le N^2\cdot (\log N)^k$, $1\le j\le 2$, and every $\bf{\e}=(\e_1,\e_2)\in\{-1,+1\}^2$, let 
$$
T(\tau; \v,\bf{\e})=\{\n\in \Z^2: 2^{v_1-1}\le \e_1 n_1 <2^{v_1}, \ 2^{v_2-1}\le \e_1 (n_1\tau-n_2) <2^{v_2}\}
$$
with the convention that if $v_1=0$ then the requirement above means $n_1=0$, and since $|n_1\tau-n_2|\ge 1/2$ we have $v_2\ge 0$. Note that $0\le v_j\ll \log N$, $1\le j\le 2$. 
Let 
$$
V_1(N)=\{\v\in \Z^2: \min_{1\le j\le 2} v_j \le c^*\log \log N \text{ with } c^* = \fc{20k}{\log 2}\},
$$
and 
$$
V_2(N)=\{\v\in \Z^2 :  \exists j\in \{1,2\} \text{ such that } \fc{N^2}{4} \le 2^{v_j}\le N^2\cdot (\log N)^k\}.
$$
Let $C=2^q$, $q$ an integer with $2^q>(\log N)^{40}=2^{c^* \log \log N}$, and write
$$
T(\tau; \v,\e,q)=\l\{ \n\in T(\tau; \v,\bf{\e}): 2^{q-1}< \max\{1,|n_1|\}|n_1\tau-n_2|\cdot \|\n\bma\|\le 2^q\r\}
$$
We use Lemma 3.2 to have the estimation for the cardinality for $T(\tau; \v,\bf{\e},q)$: for almost every $\bma\in \R^2$, since $2^q>(\log N)^{40}$, we have 
$$
\bal
|T(\tau; \v,\bf{\e},q)|\ll &\sum_{\n\in T(\tau; \v,\bf{\e})} \min\l\{1,\fc{2^q}{\max\{1,|n_1|\} |n_1\tau-n_2|}\r\}\\
&+O((2^q)^{\fc{3}{4}+\e}\cdot (\log N)^{\fc{1}{2}+\e})\ll 2^q.\\
\eal
$$
Note that $|V_j|=O(\log N\cdot \log \log N)$, $j=1,2$; and there are $O(\log N)$ choices for $q$ since $2^q\le (N^2(\log N)^2)^2$. Combined with the fact that each set of $T(\tau; \v,\bf{\e},q)$ would contribute $O(1)$, the difference can be estimated by the following:
$$
\bal
&|\bar{D}_5-\bar{D}_4|\\
&\ll \sum_{\bf{\e}\in \{-1,+1\}^2}\l(\sum_{\v\in V_1(N)}+\sum_{\v\in V_2(N)} \r)\\
&\sum_{q: 2^q\le (N^2\cdot (\log N)^2)^2} \sum_{\n\in T(\tau; \v,\bf{\e},q)} \fc{1}{\|\n\bma\|\cdot \max\{1,|n_1|\}||n_1\tau-n_2|}\\
&\ll \log N \cdot \log \log N \cdot \sum_{q: 2^q\le (N^2\cdot (\log N)^2)^2} \fc{2^q}{2^{q-1}}\\
&\ll (\log N)^2\cdot \log \log N.
\eal
$$
\end{proof}
In next section we will tackle the main difficulty for the sum of $\ftwo$ in $U_5$.

\section{Estimation of the exponential sums.}
The form of the product $\ftwo$ saves us the work for the constant part which is dealt with in section 6 of Beck\cite{Beck}, and it remains to verify that everything in Beck's proof for the exponential part also works for the shifted divisor $n_1(n_1\tau-n_2)(\n\bma-n_3)$. Since in Section 4 we already controlled the terms where $|n_1\tau-n_2|=\|n_1\tau\|$ which could be very small, the remaining sum would behave similarly as the exponential sums in the case of boxes.

To estimate the contribution of the remaining sum, first we highlight the essentials in the sum $\bar{D}_5$, which can be represented as follows:
\beq\label{multiplied out term}
\bal
\bar{D}_5(\bf{x},\bm{a},\bm{\a};N)=\fc{\i^3}{(2\pi)^3}
\Bigg(
&\sum_{\bf{s}}\pm\sum_{\bf{n}\in U_5(\n,\tau; N)} \fc{e^{2\pi \i \bf{\L_s}(\bf{n})}}{n_1(n_1\tau-n_2)(\n\bma-n_3)}\cdot g(\n, N, \bma)\Bigg)\\
\eal
\eeq
where $g(\bf{n}, N, \bma)$ is the product below (observe that $|g(\bf{n},N,\bma)|\le 1$):
\beq\label{gN}
\l(\fc{\sin 2\pi (\n\bma-n_3)}{2\pi (\n\bma-n_3)}\r)^2\cdot\prod_{j=1}^{2}\l(\fc{\sin 2\pi(\fc{n_j}{N^2})}{2\pi(\fc{n_j}{N^2})}\r)^2
\eeq
and $\L_\bf{s}=\L_{\bf{s},\bm{a}, x,N, \bma}$ is one of the 4 linear forms of $3$ variables:
\beq\label{linear forms}
\L_\bf{s}(\bf{n})=\L_\bf{s}(n_1,n_2,n_3)=n_1(\d_1x_1-a_1)+n_2(\d_2x_2-a_2) - \d_3N(\n\bma-n_3) 
\eeq
where $\bf{s}=(\d_1,\d_2,\d_3)\in \{0,1\}^3$, and $\d_1+\d_2=1$.

Note from \eqref{term} that the sign $\pm$ in \eqref{multiplied out term} is in fact $\pm=(-1)^{\d_2+\d_3+1}$, and so it is independent of $\bf{n}\in \Z^3$.

The main idea is to divide the divisors into small ranges and cancel out the positives with the negatives, with the aim of getting rid of the extra $\log N$ in the usual estmations by the Erdős–Turán inequality.  

Let $\d_N=(\log N)^{-2}$, and for every $\bf{\e}=(\e_1,\e_2,\e_3)\in \{-1,+1\}^3$ and $\bf{l}=(l_1,l_2,l_3)\in \N^3$ with 
\beq\label{range for l}
\bal
&(\log N)^{40}\le (1+\d_N)^{l_j}\le \fc{N^2}{4}, \ 1\le j\le 2, \text{ and }\\
&(\log N)^{40} \le (1+\d_N)^{l_3}\le \l(\fc{N^2}{4}\r)^2
\eal
\eeq
and write 
\beq\label{Def for Ul}
U(\bf{l},\bf{\e})=
\l\{
\n\in \Z^2 \l|
\bal 
& (1+\d_N)^{l_1}\le \e_1n_1<(1+\d_N)^{l_1+1},\\
& (1+\d_N)^{l_2}\le \e_2(n_1\tau-n_2)< (1+\d_N)^{l_2+1},\\
& (1+\d_N)^{l_2} \le \e_3n_1 (n_1\tau-n_2)(\n\bma-n_3)<(1+\d_N)^{l_3+1}
\eal
\r.\r\}
\eeq
where $\n=(n_1,n_2)$, and $n_3$ is the nearest integer to $\n\bma=n_1\a_1+n_2\a_2$. Note that by \eqref{range for l} the range for $\bf{l}$ is of order $(\log N/\d_N)^{3}$.

Let $\bfe^{+}$ and $\bfe^{-}$ be two vectors in $\{-1,+1\}^{3}$ such that the corresponding first $2$ coordinates are equal: $\e_j^{+}=\e_j^{-}$, $1\le j\le k$, but the last coordinates are different: $\e_{3}^{+}=1$, $\e_{3}=-1$. Since in Section 4 we have done all the estimations by bounding the absolute values of the terms, we can assume without loss of generality that the remaining terms in $\bar{D}_5$ form a perfect union of boxes $U(\bf{l},\bfe^+)\cup U(\bf{l},\bfe^{-})$ by attaching additional border points. Therefore we have 
\beq
\bal
&\bar{D}_5(\bf{x}, \bm{a}, \bma; N)\\
&= \sum_{s}\sum_{(\bf{l}, \bfe^{\pm})}
\l(
\sum_{\n\in U(\bf{l},\bfe^{+})} \ftwo +\sum_{\n\in U(\bf{l},\bfe^{-})} \ftwo
\r)\\
\eal
\eeq
By Lemma 3.2, for both $\bfe^{\pm}$ we have precisely the same estimation:
\begin{lem}\label{Lemma for |U(l,epsilon)|=main + Err}
For both $\bfe^{+}$ and $\bfe^{-}$, we have
\beq\label{estimation for the cardinality of U(l,e)=main term + error term}
|U(\bf{l},\bfe^{+\text{or} -})|=E(\bf{l})+\text{Err},
\eeq
where $E(\bf{l})$ is the main term, and Err is the Error term and specifically:
\beq
\bal
&E(\bf{l})\ll\d_N^{3}(1+\d_N)^{l_{3}}\\
&\text{Err}\ll \d_N^{4}(1+\d_N)^{l_{3}}
\eal
\eeq
\end{lem}
\begin{proof}
In Lemma 3.2, by taking $C_1=(1+\d_N)^{l_3+1}$, $C_2=(1+\d_N)^{l_3}$, and $V$, $W$, with 
$$
v_j=(1+\d_N)^{l_j},\  w_j=(1+\d_N)^{l_j+1}, \  1\le j\le 2,
$$
we have that the main term
$$
\bal
E(\bf{l})&=\sum_{\n\in U(\bf{l}, \bfe)} \l(\min \{\fc{1}{2},\fc{C_1}{|n_1||n_1\tau-n_2|} \}-\min\{\fc{1}{2},\fc{C_2}{|n_1||n_1\tau-n_2|}\}\r)\\
&\ll (C_1-C_2)\d_N^2\ll(1+\d_N)^{l_3}\d_N^3
\eal
$$
and the error term
$$
Err\ll C_1^{\fc{3}{4}+\e}\cdot (\log N)^{1+\e}\ll\d_N^4 (1+\d_N)^{l_3},
$$
since $C_1>(\log N)^{40}$. %\textcolor{red}{This is why Beck has chosen 40 as the power}
\end{proof}
We adopt the same definition as Beck in \cite{Beck} for the $\bfe$-big vectors, $\bfe$-neigbors and $\bfe$-lines.
\begin{Def}
We say that $\bf{l}=(l_1,l_2,l_3)\in \N^{3}$ satisfying \eqref{range for l} is an \textbf{$\bfe$-big} vector if 
\beq\label{big vector}
\fc{|U(\bf{l},\bfe^+)|+|U(\bf{l},\bfe^-)|}{\log N}\le \l|\sum_{\bf{n}\in U(\bf{l},\bfe^+)}e^{2\pi \i \L(\bf{n})}-\sum_{\bf{n}\in U(\bf{l},\bfe^-)}e^{2\pi \i \L(\bf{n})}\r|,
\eeq
where $|U|=\#U$ denotes the cardinality of the set $U$. Let $U(\bf{l}, \bfe^{\pm})=U(\bf{l}, \bfe^{+})\cup U(\bf{l},\bfe^{-})$
\end{Def}

\begin{Def}
Two integral vecotrs $\bf{l}=(l_1,l_2,l_3)$ and $\bf{h}=(h_1,h_2,h_3)$ satisfying \eqref{range for l} are called \textbf{neighbors} if 
\beq\label{epsilon neigborhood for l_1 and l_2}
(1+\d_N)^{\e_j(h_j-l_j)}=(\log N)^2, \ 1\le j\le k, \text{ and}
\eeq
\beq\label{epsilon neigborhood for l_3}
(1+\d_N)^{h_3-l_3}=(\log N)^{27}, 
\eeq
\end{Def}
The notation $\bf{l} {\ra} \bf{h}$ means that the ordered pair  $\bf{\<l, h\>}$ of vectors satisfies \eqref{epsilon neigborhood for l_1 and l_2} and \eqref{epsilon neigborhood for l_3}.

Note that by slightly modifying the value of $\d_N\approx (\log N)^{-2}$, we can make sure that the above definitions are met for integer vectors $\bf{l}$ and $\bf{h}$.

\begin{Def}
A sequence $H=\<\bf{h}^{(1)},\bf{h}^{(2)},\bf{h}^{(3)},\dots\>$ of vecotrs satisfying \eqref{range for l} is called a ``\textbf{special line}" if $\bf{h}^{(1)} {\ra}\bf{h}^{(2)} {\ra}\bf{h}^{(3)} {\ra}\dots$, that is, any two consecutive vectrs in $H$ are neighbors.
\end{Def}

\begin{lem}\label{key lemma}
For almost every $\bma$, every special line contains at most one $\bm{\e}$-big vector.
\end{lem}

\begin{proof}
Let $H=\<\bf{h}^{(1)},\bf{h}^{(2)},\bf{h}^{(3)},\dots\>$ be a special line with two $\bm{\e}$-big vectors $\bf{h}^{(p)}$ and $\bf{h}^{(q)}$, $1\le p<q$. If 
\beq\label{small L(n)}
\|\L(\bf{n})\|\le (\log N)^{-2} \text{ for every $\bf{n}\in U(\bf{h}^{(p)},\e^{\pm})$},
\eeq
then 
\beq{\label{1-exp small}}
|1-e^{2\pi \i \L(\bf{n})}|\ll (\log N)^{-2} \text{ for every $\bf{n}\in U(\bf{h}^{(p)},\e^{\pm})$}.
\eeq
By repeating the argument of the cancellation of the main term with the above equation, we obtain (Err means error for the number of elements in $S$)
$$
\bal
&\l|\sum_{\bf{n}\in U(\bf{h}^{(p)},\e^+)}e^{2\pi \i \L(\bf{n})}-\sum_{\bf{n}\in U(\bf{h}^{(p)},\e^-)}e^{2\pi \i \L(\bf{n})}\r|\\
&\ll (1+(\log N)^{-2})(E(\bf{h}^{(p)}))+|\text{Err}|)-(1-(\log N)^{-2})(E(\bf{h}^{(p)}))-|\text{Err}|)\\
&\ll |\text{Err}|+(\log N)^{-2}E(\bf{h}^{(p)}))\\
&\ll \d_N(E(\bf{h}^{(p)}))\\
&\ll \fc{|U(\bf{h}^{(p)},\e^{\pm})|}{(\log N)^2}.
\eal
$$

But this contradicts the assumption that $\bf{h}^{(p)}$ is $\bf{\e}$-big, see \eqref{big vector}. So there is an $\bf{n^*}\in U(\bf{h}^{(p)}, \bf{\e^{\pm}})$ such that
\beq\label{big L(n)}
\|\L(\bf{n})\|> (\log N)^{-2} .
\eeq
For every $\bf{m}\in U(\bf{h}^{(q)},\e^{\pm})$ (another $\bm{\e}$-big vector), consider the "arithmetic progression" with difference $\bf{n^*}$:
$$
\bf{m}+r\cdot \bf{n^*}=(m_1+r\cdot \nstar_1,m_2+r\cdot \nstar_2,m_3+r\cdot \nstar_3), \quad r=0,\pm 1,\pm 2,\dots
$$
We will estimate how many consecutive members $\bf{m}+r\cdot \bf{n^*}$ are contained in $U(\bf{h}^{q},\e^{\pm})$. Since $n^*\in \Up$, the definition for $\Uq$ (see \eqref{Def for Ul}) gives the following:
\beq\label{range for hp_1}
(1+\d_N)^{h_1^{(p)}}\le\e_1n^*_1<(1+\d_N)^{\hp_1+1},
\eeq
\beq\label{range for hp_2}
(1+\d_N)^{h_2^{(p)}}\le\e_2(n^*_1\tau-n^*_2)<(1+\d_N)^{\hp_2+1},
\eeq
\beq\label{range for hp_3}
(1+\d_N)^{h_3^{(p)}}\le |n^*_1||n^*_1\tau-n^*_2|\|\n^{*}\bma\|<(1+\d_N)^{\hp_3+1}.
\eeq
\begin{Def}\label{Def for inner points}
An $\bf{m}\in \Uq$ is called an inner point if 
\beq\label{range for inner point 1}
(1+\d_N)^{h_1^{(q)}}\l(1+\fc{\d_N}{(\log N)^2}\r)\le\e_1m_1<\l(1-\fc{\d_N}{(\log N)^2}\r)(1+\d_N)^{\hq_1+1},
\eeq
\beq\label{range for inner point 2}
(1+\d_N)^{h_2^{(q)}}\l(1+\fc{\d_N}{(\log N)^2}\r)\le\e_2(m_1\tau-m_2)<\l(1-\fc{\d_N}{(\log N)^2}\r)(1+\d_N)^{\hq_2+1},
\eeq
\beq\label{range for inner point 3}
(1+\d_N)^{h_3^{(q)}}\l(1+\fc{\d_N}{(\log N)^2}\r)\le |m_1||m_1\tau-m_2|\cdot\|\bf{m}\bma\|<\l(1-\fc{\d_N}{(\log N)^2}\r)(1+\d_N)^{\hq_3+1}.
\eeq
The rest of the points in $\Uq$ are called border points.
\end{Def}
For every inner point $\bf{m} \in \Uq$, and for every $|r|\le (\log N)^4$, it follows from \eqref{epsilon neigborhood for l_1 and l_2}, \eqref{range for hp_1} and \eqref{range for inner point 1} that, 
\beq\label{range for coordinate 1 of progression}
\bal
(1+\d_N)^{h_1^{(q)}}&<(1+\d_N)^{h_1^{(q)}}\l(1+\fc{\d_N}{(\log N)^2}\r)-(\log N)^4(1+\d_N)^{\hp_1+1}\\
&\le\e_1(m_1+r\cdot n^*_1)\\
&<\l(1-\fc{\d_N}{(\log N)^2}\r)(1+\d_N)^{\hq_1+1}+(\log N)^4(1+\d_N)^{\hp_1+1}\\
&<(1+\d_N)^{h_1^{(q)}+1}.\\
\eal
\eeq
Similarly, from \eqref{epsilon neigborhood for l_1 and l_2}, \eqref{range for hp_2} and \eqref{range for inner point 2}, we obtain the following, 
\beq\label{range for coordinate 2 of progression}
(1+\d_N)^{-h_2^{(q)}}<\e_2((m_1+r\cdot n^*_1)\tau-(m_2+rn_2^*))<(1+\d_N)^{-h_2^{(q)}+1},
\eeq
and from \eqref{epsilon neigborhood for l_3}, \eqref{range for hp_3} and \eqref{range for inner point 3} we have:
\beq\label{range for coordinate 3 of progression}
\bal
&\fc{(1+\d_N)^{\hq_3}}{|m_1+r\cdot n_1^*||(m_1+r\cdot n_1^*)\tau-(m_2+r\cdot n_2^*)|}\\ 
&<\|\bf{m}\bma\|-|r|\cdot \|\n^*\bma\|\\
&\le\|(\bf{m}+r\cdot \bf{n}^*)\bma\|\\
&\le \|\bf{m}\bma\|+|r|\cdot \|\n^*\bma\|\\
&<\fc{(1+\d_N)^{\hq_3+1}}{|m_1+r\cdot n_1^*||(m_1+r\cdot n_1^*)\tau-(m_2+r\cdot n_2^*)|},\\ 
\eal
\eeq
where we use the 
In view of \eqref{range for coordinate 1 of progression}-\eqref{range for coordinate 3 of progression}, for any inner point $\bf{m}\in \Uq$, at least $(\log N)^4$ consecutive members in the progression $\bf{m}+r\cdot\bf{n}^*$ are contained in $\Uq$. Therefore, we can decommpose $\Uq$ into three parts: 
\beq\label{decomposition for Sq}
\Uq=AP^+\cup AP^-\cup BP
\eeq
where $AP^{\pm}$ denotes the family of arithmetic progreesions $\{\bf{m}+r\cdot \bf{n}^*: 0\le r \le l-1 \}$ in $\Uqplus$ and $\Uqminus$ respectively, where $l=l(\bf{m})$ is the length of the progression starting from $\bf{m}$, and $l\ge (\log N)^4$. $BP$ denotes a set of border points of $\Uq$ that are not included in any arithmetic progressions. Using $\|\mathcal{L}(\bf{n})\|>(\log N)^{-2}$ (see \eqref{big L(n)}), the linearity of $\L$, we obtain
\beq\label{control for the arithmetic progressions in Spplus}
\bal
|\sum_{\bf{AP}^+} e^{2\pi \i \L(\bf{n})}|
& \le\sum_{\text{arithmetic progressions}} |\sum_{r=0}^{l-1} e^{2\pi \i \L(\bf{m}+r\bf{n^*})}|\\
& =\sum_{\text{arithmetic progressions}} |\sum_{r=0}^{l-1} e^{2\pi \i \L(\bf{m})+r\L(\bf{n^*})}|\\
&\ll \sum_{\text{arithmetic progressions}} \fc{1}{\|\Ln\|}<\sum_{\text{arithmetic progressions}}(\log N)^2\\
&\le \sum_{\text{arithmetic progressions}} \fc{\text{length}}{(\log N)^2}\le \fc{|\Uqplus|}{(\log N)^2}\\
\eal
\eeq
since each length $\ge (\log N)^4$. Similarly, 
\beq\label{control for the arithmetic progressions in Spminus}
|\sum_{\bf{AP}^-} e^{2\pi \i \L(\bf{n})}|\ll \fc{|\Uqminus|}{(\log N)^2}
\eeq
Finally, for border points, at least one of the inequalities in definition \ref{Def for inner points} is violated, thus the range for at least one components is shrinked by the ratio $\fc{\d_N}{(\log N)^2}$. Using the same reasoning as in Lemma \ref{Lemma for |U(l,epsilon)|=main + Err}, and the fact that the cardinality of the set $BP$ can be controlled by the total number of border points of $\Uq$, we have for almost every $\bma$,
\beq\label{cardinality for border points}
|BP|\ll \fc{|\Uq|}{(\log N)^2}
\eeq
Combining \eqref{decomposition for Sq}-\eqref{cardinality for border points}, for almost every $\bma$, we obtain
$$
\l|\sum_{\bf{n}\in U(\bf{h}^{(p)},\e^+)}e^{2\pi \i \L(\bf{n})}-\sum_{\bf{n}\in U(\bf{h}^{(p)},\e^-)}e^{2\pi \i \L(\bf{n})}\r| \ll \fc{|U(\bf{h}^{(p)},\e^{\pm})|}{(\log N)^2}
$$
which contradicts the assumption that $\hq$ is $\bm{\e}$-big (see \eqref{big vector}), therefore for almost very $\bma$, every special line contains at most one $\bm{\e}$-big vector, which proves the lemma.
\end{proof}

\begin{cor}\label{number of e-big vectors}
The number of $\bm{\e}$-big vectors is $\ll\d_N^{-3}\cdot(\log N)^2(\log \log N)$.
\end{cor}{}
\begin{proof}
First we estimate the number of maximal special lines. Let $H=\<\bf{h}^{(1)},\bf{h}^{(2)},\bf{h}^{(3)},\dots\>$ be a "special line", here $\bf{h}^{(1)}$ is the first element of $H$, that is, if $\bf{h}^{(0)} {\ra}\bf{h}^{(1)}$ holds for some $\bf{h}^{(0)}$, then at least one of the following inequalities is violated(see \eqref{range for l}:
$$
(1+\d_N)^{h_j^{(0)}}\ge (\log N)^{40}, \ 1\le j\le 2 $$
$$
(1+\d_N)^{h_3^{(0)}}\ge (\log N)^{40}.
$$
Thus, by \eqref{epsilon neigborhood for l_1 and l_2} and \eqref{epsilon neigborhood for l_3}, one of the inequalities below holds:
$$
(\log N)^{40}\le(1+\d_N)^{\e_j h_1^{(1)}}\le (\log N)^{40}, \ 1\le j\le 2 $$
$$
(\log N)^{40}\le(1+\d_N)^{h_3^{(1)}}\le (\log N)^{40+27}. $$

So at least one coordinate $h_j^{(1)}$ of the first element $\bf{h}^{(1)}$ of $H$ is restricted to an interval of length const$\cdot \log \log N\cdot \d_N^{-1}$, the rest, by the condition \eqref{range for l}, are restricted to an interval of length const$\cdot \log N\cdot \d_N^{-1}$. Since the starting vector determines the whole $\bm{\e}$-special line, the number of $\bm{\e}$-special line is
$$
\ll (\log \log N)\cdot (\log N)^2\cdot\d_N^{-3}.
$$
By Lemma \ref{key lemma}, the total number of $\bm{\e}$-big vectors is also
$$
\ll (\log \log N)\cdot (\log N)^2\cdot\d_N^{-3}.
$$

\end{proof}

With the help of Lemma \ref{key lemma}, we can estimate the contribution of the exponential terms, we have the following:
\begin{Prop}\label{Estimation of the exponential terms}
For almost every $\bma$ we have 
$$
\l|\bar{D}_5\r|\ll (\log N)^2(\log \log N)
$$
\end{Prop}
\begin{proof}
$$
\bar{D}_5=\sum_{\text{small}}+\sum_{\text{big}}
$$
where 
\beq\label{small sum}
\sum_{\text{small}}=\sum_{\bf{\e}}\sum_{\text{$\bf{l}$ not $\bf{\e}$-big}}\sum_{\bf{n}\in U(\bf{l},\bf{\e})}   \ftwo
\eeq
\beq\label{big sum}
\sum_{\text{small}}=\sum_{\bf{\e}}\sum_{\text{$\bf{l}$ is $\bf{\e}$-big}}\sum_{\bf{n}\in U(\bf{l},\bf{\e})}   \ftwo
\eeq
By Lemma \ref{Lemma for |U(l,epsilon)|=main + Err}, range for $\bf{l}$: \eqref{range for l}, and \eqref{big vector},
$$
\bal
\l|\sum_{\text{small}}\r|&\ll \sum_{\bf{\e}}\sum_{\bf{l}\text{ is not $\bf{\e}$-big}} (1+\d_N)^{-l_3}\cdot\fc{ E(\bf{l})+Err}{\log N}\\
&\ll\sum_{\bf{l}:\eqref{range for l}}\fc{\d_N^3}{\log N}\\
&\ll \fc{\d_N^3}{\log N} \cdot \fc{(\log N)^3}{\d_N^3}\\
&\ll (\log N)^2\\
\eal
$$
By Corollary \ref{number of e-big vectors},
$$
\bal
|\sum_{\text{big}}|&\ll \sum_{\text{$\bf{l}$ is $\bm{\e}$-big}} (1+\d_N)^{-l_3}\cdot (E(\bf{l})+Err)\\
&\ll[(\log \log N)(\log N)^2\d_N^{-3}]\cdot[(1+\d_N)^{-l_3}\cdot (1+\d_N)^{l_3}\d_N^3]\\
&\ll(\log N)^2\log \log N\\
\eal
$$
\end{proof}
Combining Proposition 2.1, 4.1, 5.1 and Lemma 3.1, 4.1, we finally arrive at the convergent part of our main Theorem \ref{main result}. The divergent part is easy:

\nid\textbf{Proof of the divergent part of Theorem 1.2}

One way is to employ the same strategy as in Beck\cite{Beck}, and try to find a ``large'' Fourier coefficient using Lemma 3.3. But an easier way is to use Beck's result directly, since Beck proved that for $\sum_{n=1}^{\infty}\fc{1}{\varphi(n)}=\infty$, the maximal discrepancy function for boxes will be greater than $(\log N)^2\varphi(\log \log N)$ infinitely often, by cutting the box diagonally we know that at least the discrepancy function for one of the triangles is greater than $1/2(\log N)^2\varphi(\log \log N)$, since the maximal descrepancy is defined over all starting points $\bm{a}\in [0,1]^2$, we can translate the triangle with the large discrepancy to the origin, hence the maximal discrepancy defined as in the theorem 1.2 is also greater than $1/2(\log N)^2\varphi(\log \log N)$ infinitely often. Note that the coefficient $1/2$ does not change the convergence of the series of $\sum_{n=1}^{\infty}\fc{1}{\varphi(n)}$. The proof for the main theorem 1.2 is completed. \qed

\clearpage

\newpage 
\bibliographystyle{plain}
\bibliography{references}
\end{document}